\documentclass[a4paper,reqno, 11pt]{article}
\usepackage{mathrsfs}
\usepackage{amssymb}
\usepackage[usenames,dvipsnames,svgnames,table]{xcolor}
\usepackage{url}

\usepackage[T1]{fontenc}
\makeatletter
\def\thm@space@setup{
\thm@preskip=2mm
\thm@postskip=0mm
}
\makeatother

\let\leq\leqslant
\let\geq\geqslant

\let\epsilon\varies

\usepackage[centertags]{amsmath}
\usepackage{amsfonts}
\usepackage{amsthm}
\usepackage{times}

\usepackage{hyperref}
\usepackage[capitalise]{cleveref}

\hypersetup{
    colorlinks,
    linkcolor={RoyalBlue},
    citecolor={OliveGreen},
    urlcolor={blue!80!black}
}
\colorlet{best}{red!70!blue!110!violet!110!white!100}

\usepackage{enumitem}
\usepackage{enumerate} %
\usepackage{xcolor}
\usepackage[margin=1.25in]{geometry}
\usepackage{centernot}
\usepackage[symbol]{footmisc}
\usepackage{todonotes}

\usepackage{tabularx}

\definecolor{myblue1}{RGB}{30, 144, 255}
\definecolor{myblue2}{RGB}{0, 135, 245}
\definecolor{myblue3}{RGB}{0, 126, 235}
\definecolor{myblue4}{RGB}{36, 70, 119}

\newtheorem{theorem}{Theorem}
\newtheorem*{theorem*}{Theorem}

\newtheorem{lemma}[theorem]{Lemma}

\newtheorem{example}[theorem]{Example}

\theoremstyle{definition}

\usepackage{mathtools}

\begin{document}

\title{On triangle-free graphs maximizing embeddings of bipartite graphs\thanks{This work was supported by the National Science Centre grant 2021/42/E/ST1/00193.}}

\author{Dmitriy Gorovoy\thanks{Faculty of Mathematics and Computer Science, Jagiellonian University, {\L}ojasiewicza 6, 30-348 Krak\'{o}w, Poland. E-mail: {\tt dimgor2003@gmail.com}.} \and
Andrzej Grzesik\thanks{Faculty of Mathematics and Computer Science, Jagiellonian University, {\L}ojasiewicza 6, 30-348 Krak\'{o}w, Poland. E-mail: {\tt Andrzej.Grzesik@uj.edu.pl}.} \and 
Justyna Jaworska\thanks{Faculty of Mathematics and Computer Science, Jagiellonian University, {\L}ojasiewicza 6, 30-348 Krak\'{o}w, Poland. E-mail: {\tt justynajoanna.jaworska@student.uj.edu.pl}.}}

\date{}

\maketitle
 	
\begin{abstract}
In 1991 Gy\H ori, Pach, and Simonovits proved that for any bipartite graph $H$ containing a matching avoiding at most 1 vertex, the maximum number of copies of $H$ in any large enough triangle-free graph is achieved in a balanced complete bipartite graph.
In this paper we improve their result by showing that if $H$ is a bipartite graph containing a matching of size $x$ and at most $\frac{1}{2}\sqrt{x-1}$ unmatched vertices, then the maximum number of copies of~$H$ in any large enough triangle-free graph is achieved in a complete bipartite graph.
We also prove that such a statement cannot hold if the number of unmatched vertices is $\Omega(x)$. 
\end{abstract}
	
\section{Introduction}

A classical theorem of Tur\'an \cite{turan_thm} states that the unique $K_r$-free graph on $n$ vertices with the maximum number of edges is the balanced complete $(r-1)$-partite graph, denoted by $T_{r-1}(n)$. 
This result was further generalized by Zykov \cite{zykov_cliques} (and independently by Erd\H os \cite{erdos1962}), who proved that among $K_r$-free $n$-vertex graphs, also $T_{r-1}(n)$ maximizes the number of copies of any complete graph $K_s$ for $s < r$. 
In general, the maximum number of copies of a given graph~$H$ among all $K_r$-free $n$-vertex graphs (for $r>\chi(H)$) is not always achieved in $T_{r-1}(n)$.
For example if $H$ is a star on 4 vertices and $r=3$.
Nevertheless, recently, Morrison, Nir, Norin, Rz{\k{a}}{\.z}ewski and Wesolek~\cite{morrison2023}, answering a conjecture of Gerbner and Palmer \cite{gerbner2022}, showed that for any graph~$H$ the maximum number of copies of $H$ in a large enough $K_{r}$-free $n$-vertex graph is obtained in $T_{r-1}(n)$ as long as $r$ is large enough. 

A natural generalization of the previously mentioned results is to search for sufficient conditions for the maximum number of copies of a given graph~$H$ in a $K_r$-free $G$ to be maximized when $G$ is some (not necessarily balanced) complete $(r-1)$-partite graph. 
Note that an easy application of the graph removal lemma implies that if $\chi(H) < \chi(F)$ then the maximum number of copies of $H$ in $F$-free graphs is asymptotically the same as in $K_{\chi(F)}$-graphs, so solving the problem for forbidden complete graph is essentially solving it in the more general case as well. 

One necessary condition to have the maximum number of copies of $H$ achieved in a complete $(r-1)$-partite graph is $\chi(H) < r$. However, this is not sufficient as shown by the following example. 

\begin{example}[Gy{\H o}ri, Pach, Simonovits \cite{gps}] \label{example: gps}
Let $H$ be a bipartite graph on $2k$ vertices formed by two disjoint stars $K_{1, k - 2}$ with centers connected by a path of length 3. The number of copies of $H$ in any $n$-vertex bipartite graph is maximized in $T_2(n)$, since $H$ has the same number of vertices in both color classes. However, for sufficiently large $k$ and $n$, the number of copies of $H$ in $T_2(n)$ is significantly smaller than in a blow-up of a five-cycle with blobs of sizes $\frac{n}{2k}, \frac{n}{2k}, \frac{n}{2k}, \frac{n}{2k}$, and $n - \frac{2n}{k}$.
\end{example}

For a similar example for higher values of $r$, see \cite{counter_example}.

In the most interesting case $r = 3$, asking when the maximum number of copies of a given bipartite graph in triangle-free graphs is achieved in a complete bipartite graph, Gy{\H o}ri, Pach, and Simonovits proved the following sufficient condition.

\begin{theorem}[Gy{\H o}ri, Pach, and Simonovits \cite{gps}]
Let $H$ be a bipartite graph on $m$ vertices containing a matching of size $\lfloor \frac{m}{2} \rfloor$. Then, for $n > m$, $T_2(n)$ is the unique $n$-vertex triangle-free graph maximizing the number of copies of $H$.
\end{theorem}

Intuitively, if $H$ contains a large matching, then $T_2(n)$ is the best choice for a triangle-free maximizer, because it contains the biggest number of copies of the matching itself. The condition that $H$ must have a perfect matching (or an almost perfect matching if $2\nmid m$) cannot be relaxed to a matching on $m-2$ vertices if we want to have $T_2(n)$ as the maximizer. For example the maximum number of copies of a star $K_{1,3}$ in triangle-free graphs is not achieved in $T_2(n)$, but in a non-balanced complete bipartite graph. 

We show that the maximizer is a complete bipartite graph even using a far weaker condition on the size of a matching in $H$.

\begin{theorem} \label{th: quadratic}
Let $H$ be a bipartite graph containing a matching of size $x$ and at most $\frac{1}{2} \sqrt{x-1}$ unmatched vertices. Then, for $n$ sufficiently large, a complete bipartite graph maximizes the number of copies of $H$ among all triangle-free $n$-vertex graphs.
\end{theorem}

We do not know whether the bound on the number of unmatched vertices $O(\sqrt{x})$ is optimal, but we show that it needs to be sublinear.

\begin{theorem} \label{th: linear}
For any constant $\lambda > 0$ and integer $n_0 > 0$, there exist an integer $x$,
\begin{itemize}[leftmargin=10pt]
\setlength{\itemsep}{0pt}
\setlength{\parskip}{0pt}
\vspace{-\topsep}
\item 
a bipartite graph $H$ containing a matching of size $x$ and at most $\lambda x$ unmatched vertices, and
\item 
a triangle-free non-bipartite graph $G$ on more than $n_0$ vertices,
\end{itemize}\vspace{-\topsep}
such that the number of copies of $H$ in $G$ is larger than in any bipartite graph on the same number of vertices.
\end{theorem}

\section{Proofs}

We start with introducing the needed notation and some preliminary results. 

By the number of copies of $H$ in $G$, denoted by $H(G)$, we mean the number of subgraphs of~$G$ isomorphic to $H$. For easier calculations we consider injective embeddings of $H$ in $G$, i.e., injective functions $\varphi: V(H) \rightarrow V(G)$ such that $\varphi(v_1)\varphi(v_2) \in E(G)$ if $v_1v_2 \in E(H)$. We denote the number of different injective embeddings of $H$ in $G$ by $\overline{H(G)}$. Observe that $\overline{H(G)}$ differs from $H(G)$ just by the number of authomorpisms of $H$, so maximizing $\overline{H(G)}$ is equivalent to maximizing $H(G)$.

For two graphs $H$ and $G$ we define the \emph{$H$-degree} of a vertex $v \in V(G)$, denoted $h(v)$, as the number of injective embeddings of $H$ in $G$ whose image contains $v$. 
Analogously, for $u, v \in V(G)$ we define $h(u,v)$ as the number of injective embeddings whose image contains vertices $u$ and $v$, and $h(u, \bar{v})$ as the number of injective embeddings whose image contains vertex $u$ and does not contain vertex $v$.

\begin{lemma} \label{le: h_u,v}
For an $m$-vertex graph $H$ let $G$ be a triangle-free $n$-vertex graph that maximizes the number of injective embeddings of $H$. Then for any two vertices $u,v \in V(G)$ it holds $h(v) \leq h(u)+O\left(n^{m-2}\right)$.
\end{lemma}

\begin{proof}
Modify the graph $G$ by deleting $u$ and adding  instead a copy of $v$ (not adjacent to~$v$). The obtained graph remains triangle-free after such modification. In this process we lose $h(u)$ injective embeddings and gain $h(v,\bar{u})$ new ones. Since $h(v)=h(v,u)+h(v, \bar{u})$, we get $h(v, \bar{u})-h(u) =h(v)-h(u,v)-h(u) \leq 0$, so $h(v) \leq h(u) + h(u, v) = h(u)+O\left(n^{m-2}\right)$. 
\end{proof}

\begin{lemma} \label{le: max edges}
Let $G$ be a triangle-free graph on $n$ vertices with the maximum degree $\Delta$. Then $|E(G)| \leq \Delta(n- \Delta)$.
\end{lemma}
\begin{proof}
    Consider a vertex of degree $\Delta$ and let $A$ be the set of its neighbors. The set $V(G)\setminus A$ contains $n-\Delta$ vertices of degree at most $\Delta$. Moreover, from triangle-freeness of $G$ there are no edges between vertices in $A$. Thus, $|E(G)| \leq \Delta(n-\Delta)$.

    Note that the equality holds only for the complete bipartite graph $K_{\Delta, n-\Delta}$.
\end{proof}


We are ready to prove the main theorem.

\begin{proof}[Proof of Theorem \ref{th: quadratic}]
If $H$ contains an isolated vertex, then the graph $H'$ obtained by removing it satisfies the assumptions of Theorem~\ref{th: quadratic} (with a smaller number of unmatched vertices). Moreover, for every $n$-vertex graph $G$ we have $\overline{H(G)}=\overline{H'(G)} (n-|V(H')|)$. Thus, if the theorem holds for $H'$ then it also holds for $H$.
Therefore, we may assume that $H$ does not contain isolated vertices. 
We may also assume that the matching of size $x$ is a maximal matching in $H$ and $x \geq 2$. 

Let $m$ be the number of vertices in $H$ and $c$ be the number of connected components of $H$. 
For a sufficiently large~$n$ let $G$ be a triangle-free graph on $n$ vertices that has the largest number of copies of $H$. 
By summing up the $H$-degrees of all vertices in $G$ we count each injective embedding exactly $m$ times, so 
\[\sum_{v \in V(G)} h(v)=m \cdot \overline{H(G)}\geq m \cdot\overline{H(T_2(n))} = m \cdot 2^c \left(\frac{n}{2}\right)^{m} + O\left(n^{m-1}\right).\]
The last equality holds because we can count the number of injective embeddings of $H$ in~$T_2(n)$ by embedding each vertex connected to already embedded vertices in $n/2$ ways and a vertex in a new component of $H$ in $n$ ways. 
The lower order error term comes from the possibility of selecting the same vertex multiple times. 

For any vertex $u \in V(G)$ by summing up the inequality from Lemma~\ref{le: h_u,v} for each vertex $v \in V(G)$, and combining it with the above bound, we obtain 
\[n \cdot h(u) \geq \sum_{v \in V(G)} h(v) + O(n^{m-1}) \geq m \cdot 2^{c} \left(\frac{n}{2}\right)^{m} + O\left(n^{m-1}\right).\]
Therefore, it holds
\begin{equation}\label{eq:lower}
h(u) \geq m 2^{c-1}\left(\frac{n}{2}\right)^{m-1} + O\left(n^{m-2}\right).
\end{equation}

Our plan now is to upper-bound $h(u)$ for a vertex $u$ of the minimum degree $\delta$ in $G$. 
Consider an arbitrary vertex $w \in V(H)$.
We estimate the number of injective embeddings that map $w$ to~$u$ in the following way.
As $H$ does not contain isolated vertices and the matching of size $x$ is maximal, $w$ is adjacent to some matched vertex. 
It implies that there exists a matching $M$ in~$H$ of size $x$, which contains $w$. 
Thus, we have $\delta$ ways to embed in $G$ the neighbor of $w$ in the matching $M$. Then, we have at most $|E(G)|$ ways to embed in $G$ any edge of $M$ from the same component as already embedded vertices, and at most $2|E(G)|$ ways for each edge in a new component. Finally, we have at most $\Delta^{m-2x}$ ways to embed all vertices of $H$ not belonging to~$M$, where $\Delta$ is the maximum degree of $G$.
Therefore, taking into account that we can choose $w$ in $V(H)$ in $m$~ways and applying Lemma \ref{le: max edges} to bound the number of edges of $G$, we obtain
\begin{equation}\label{eq: upper}
 h(u)\leq m \delta 2^{c-1}|E(G)|^{x-1} \Delta^{m-2x}\leq m \delta  2^{c-1} \Delta^{m-x-1}(n-\Delta)^{x-1}. 
\end{equation}

By combining inequalities (\ref{eq:lower}) and (\ref{eq: upper}) we conclude the following bound for $\delta$
\begin{equation}\label{eq: delta}
 \delta \geq \frac{(\frac{n}{2})^{m-1} }{\Delta^{m-x-1}(n-\Delta)^{x-1}} + O(1).
\end{equation}

Our goal is to show that under the assumptions of Theorem~\ref{th: quadratic} from inequality (\ref{eq: delta}) we derive that $\delta > \frac{2}{5}n$. Then, since Andr{\'a}sfai-Erd{\H o}s-S{\'o}s \cite{andrasfai-erdos-sos} theorem gives that every $n$-vertex triangle-free graph with minimum degree greater than $\frac{2}{5}n$ is bipartite, the graph $G$ will be bipartite. 

For convenience, we replace $m-2x$ with $2d$. 
It is straightforward to check that the denominator $\Delta^{m-x-1}(n-\Delta)^{x-1} = \Delta^{2d+x-1}(n-\Delta)^{x-1}$ is maximized for
$$\Delta=\frac{2d+x-1}{2d+2x-2}n.$$
Thus, from (\ref{eq: delta}) we obtain
\[\delta \geq \frac{(d+x-1)^{2d+2x-2}}{2(2d+x-1)^{2d+x-1}(x-1)^{x-1}} n + O(1).\]

Note that it is enough to show that 
\[\frac{(d+x-1)^{2d+2x-2}}{2(2d+x-1)^{2d+x-1}(x-1)^{x-1}} > \frac{2}{5}\]
as then, for large enough~$n$, we have the wanted inequality $\delta > \frac{2}{5}n$.  

Rearranging the terms and using $2d \geq 0$, Bernoulli's inequality and $d^2 \leq \frac{1}{16}(x-1)$ implied by assumptions of the theorem we obtain
\begin{align*}
    & \frac{(d+x-1)^{2d+2x-2}}{2(2d+x-1)^{2d+x-1}(x-1)^{x-1}} \\
    & = \frac{1}{2} \left(1- \frac{d}{2d+x-1}\right)^{2d+x-1}\left(1+ \frac{d}{x-1}\right)^{x-1}\\
    & \geq \frac{1}{2} \left(1- \frac{d}{x-1}\right)^{2d+x-1}\left(1+ \frac{d}{x-1}\right)^{x-1}\\
    & = \frac{1}{2}\left(1-\frac{d}{x-1}\right)^{2d}\left(\left(1-\frac{d}{x-1}\right)\left(1+ \frac{d}{x-1}\right)\right)^{x-1}\\
    & = \frac{1}{2}\left(1-\frac{d}{x-1}\right)^{2d}\left(1-\frac{d^2}{(x-1)^2}\right)^{x-1}\\
    & \geq \frac{1}{2}\left(1-\frac{2d^2}{x-1}\right)\left(1-\frac{d^2}{x-1}\right)\\
    & \geq \frac{1}{2}\left(1-\frac{2}{16}\right)\left(1-\frac{1}{16}\right) \\
    & > \frac{2}{5}
\end{align*}
as needed.
\end{proof}

\begin{proof}[Proof of Theorem \ref{th: linear}]
It is enough to consider rational $\lambda$. Let $d \geq 1$ and $x \geq 3$ be large enough integers such that $\lambda = \frac{2d}{x}$. 
We modify Example~\ref{example: gps}. Consider a graph~$H$ consisting of two stars having $d+1$ leaves with centers connected by a path of length 3 and additional $x-3$ paths of length 2 starting in a central vertex from the aforementioned path, see Figure \ref{fig: counterexample}. 
Note that $H$ has a matching of size $x$ and $2d = \lambda x$ unmatched vertices.

\begin{figure}[ht]
\centering
\begin{tikzpicture}[color=best, x=0.6cm,y=0.35cm, line width=0.02cm]
\draw (6,0)--(-4,0);
\draw (0,0)--(-2,2)--(-2,4);
\draw (0,0)--(-1,2)--(-1,4);
\draw (0,0)--(0,2)--(0,4);
\draw (0,0)--(2,2)--(2,4);
\draw [fill=best] (4,0) circle (2pt);
\draw [fill=best] (2,0) circle (2pt);
\draw [fill=best] (0,0) circle (2pt);
\draw [fill=best] (-2,0) circle (2pt);

\draw [fill=white] (-4, 0) ellipse (13pt and 8pt);
\draw [fill=white] (6, 0) ellipse (13pt and 8pt);

\draw [fill=best] (0,2) circle (2pt);
\draw [fill=best] (-1,2) circle (2pt);
\draw [fill=best] (-2,2) circle (2pt);
\draw [fill=best] (2,2) circle (2pt);

\draw [fill=best] (0,4) circle (2pt);
\draw [fill=best] (-1,4) circle (2pt);
\draw [fill=best] (-2,4) circle (2pt);
\draw [fill=best] (2,4) circle (2pt);

\node at (1,2) {$\dots$};

\begin{scriptsize}
\node[color=black][scale=1.25]  at (6,0) {$d+1$};
\node[color=black][scale=1.25]  at (-4,0) {$d+1$};
\node[color=black][scale=1.25]  at (0,5) {$x-3$};
\node[color=black, below][scale=1.25]  at (-4,-0.6) {$1$};
\node[color=black, below][scale=1.25]  at (6,-0.6) {$1$};
\node[color=black, below][scale=1.25]  at (4,0) {$5$};
\node[color=black, below][scale=1.25]  at (2,0) {$4$};
\node[color=black, below][scale=1.25]  at (-2,0) {$2$};
\node[color=black, below][scale=1.25]  at (0,0) {$3$};
\node[color=black, left][scale=1.25]  at (-2,2) {$2$};
\node[color=black, left][scale=1.25]  at (-1,2) {$2$};
\node[color=black, left][scale=1.25]  at (0,2) {$2$};
\node[color=black, right][scale=1.25]  at (2,2) {$2$};
\node[color=black, left][scale=1.25]  at (-2,4) {$1$};
\node[color=black, left][scale=1.25]  at (-1,4) {$1$};
\node[color=black, left][scale=1.25]  at (0,4) {$1$};
\node[color=black, right][scale=1.25]  at (2,4) {$1$};
\end{scriptsize}
\end{tikzpicture}
\hspace{2cm}
\begin{tikzpicture}[scale=0.8,color=best,line width=0.03cm,
block/.style={circle, fill=white,draw=best,inner sep=5pt}]
    \foreach \i in {1,...,5}
    {
        \coordinate (\i) at (\i*360/5:1.5);
    }
\draw (1)--(2)--(3)--(4)--(5)--(1);
\foreach \i in {1,...,5}
    {
        \node[block] at (\i*360/5:1.5) {};
    }
\begin{scope}
\node [color=black, label={[label distance=0cm, color=black]0:$1$}] at (1) {$a$}; 
\node [color=black, label={[label distance=-0.1cm, color=black]30:$2$}] at (2) {$b$}; 
\node [color=black, label={[label distance=-0.1cm, color=black]10:$3$}] at (3) {$c$}; 
\node [color=black, label={[label distance=0cm, color=black]0:$4$}] at (4) {$c$}; 
\node [color=black, label={[label distance=0cm, color=black]0:$5$}] at (5) {$c$}; 

\end{scope}

\end{tikzpicture}
\caption{Graphs $H$ and $G$.}
\label{fig: counterexample}
\end{figure}
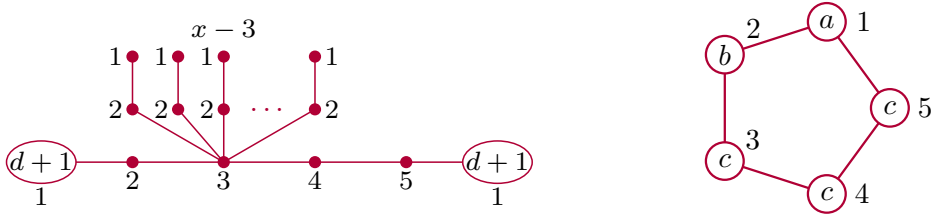

We will show that if $x$ is large enough then $\overline{H(T_2(n))} < \overline{H(G)}$, where $G$ is an unbalanced blow-up of a $C_5$ with parts of sizes $a n, b n, c n, c n, c n$ for large enough $n$ and some values of $a$, $b$, $c$ with $a+b+3c=1$ to be established later. 
Since $H$ is balanced, $T_2(n)$ maximizes the number of injective embeddings of~$H$ among all bipartite graphs, and so the inequality $\overline{H(T_2(n))} < \overline{H(G)}$ implies that no bipartite graph on $n$~vertices contains more copies of $H$ than $G$.

Since $H$ is a connected balanced bipartite graph, $\overline{H(T_2(n))} \leq 2(\frac{1}{2})^{2x+\lambda x}n^{2x+\lambda x}$. 

On the other hand, to lower bound $\overline{H(G)}$ we count the number of embeddings where each vertex of $H$ is embedded to a specific blob of $G$ marked in Figure~\ref{fig: counterexample}
\begin{align*}
\overline{H(G)} & \geq a^{x+\lambda x-1} b^{x-2} c^3 n^{2x+\lambda x} + O(n^{2x+\lambda x-1}) \\
& = a^{x+\lambda x-1} (1-a-3c)^{x-2} c^3 n^{2x+\lambda x} + O(n^{2x+\lambda x-1}).
\end{align*} 

Note that for arbitrary $\lambda > 0$ it is possible to choose $a>\frac{1}{2}$ such that
$$a^{\lambda + 1} (1 - a) - \left( \frac{1}{2} \right)^{\lambda + 2} > 0.$$
This is due to the fact that the function $f(z) = z^{\lambda+1} (1-z) - \left(\frac{1}{2}\right)^{2+\lambda}$ has a root in $\frac{1}{2}$ and its derivative at $\frac{1}{2}$ is greater than zero. 

Now consider the function $g(z)=a^{\lambda + 1} (1-a-3z)-\left(\frac{1}{2}\right)^{2+\lambda}$. It is continuous at $z=0$ and $g(0)=f(a)>0$. Thus, we can choose small $c>0$ so that $g(c)>0$. For such $c$ define
$$
p = \frac{a^{\lambda + 1} (1 - a - 3c)}{\left( \frac{1}{2} \right)^{\lambda + 2}} > 1
$$
and take $x > \log_{p}\left(\frac{2 a (1-a-3c)^2}{c^3}\right)$. Then
$$\frac{a^{x + \lambda x} (1-a-3c)^{x}}{\left(\frac{1}{2}\right)^{2x+2\lambda}} > \frac{2 a (1-a-3c)^2}{c^3},$$
which implies
$$a^{x + \lambda x-1} (1-a-3c)^{x-2}c^3 > 2\left(\frac{1}{2}\right)^{2x+2\lambda}.$$
This means that for large enough $n$ we have 
$\overline{H(G)} \geq \overline{H(T_2(n))}$ as needed.
\end{proof}

	\bibliographystyle{abbrv}
	\bibliography{bib_source}

\end{document}